\documentclass[10pt]{amsart}

\usepackage{amsmath}
\usepackage{amsfonts}
\usepackage{amssymb}
\usepackage{amsthm}
\usepackage{upgreek}
\usepackage{bigdelim}
\usepackage{multirow}
\usepackage{hhline}
\usepackage{enumerate}
\usepackage{comment}
\usepackage{amsthm,verbatim}
\usepackage{latexsym}
\usepackage{fancyhdr}

\newcommand{\Z}{\mathbb{Z}}
\newcommand{\N}{\mathbb{N}}
\newcommand{\R}{\mathbb{R}}

\newcommand{\ve}{\varepsilon}
\renewcommand{\mod}[1]{{\ifmmode\text{\rm\ (mod~$#1$)}\else\discretionary{}{}{\hbox{ }}\rm(mod~$#1$)\fi}}

\usepackage{pgf}
\usepackage{tikz}
\usetikzlibrary{automata,matrix,arrows,decorations.pathmorphing,decorations.pathreplacing}

\renewcommand{\Lambda}{\Uplambda}

\newtheorem{theorem}{Theorem}[section]
\newtheorem{lemma}[theorem]{Lemma}

\newtheorem{corollary}[theorem]{Corollary}
\newtheorem{proposition}[theorem]{Proposition}

\begin{document}






\title[Infinite Transformations  ] {Strict Doubly Ergodic Infinite Transformations }

\author[Isaac Loh]{Isaac Loh}

\address{ }

\email{isaaccloh@gmail.com}

\email{ }

\author[C. E. Silva]{Cesar E. Silva}

\address{Department of Mathematics and Statistics, Williams College, 
\\ Williamstown, MA 01267} 

\email{csilva@williams.edu}


\begin{abstract}
We give examples of  rank-one transformations that are (weak) doubly ergodic and rigid (so all their cartesian products are conservative), but with non-ergodic $2$-fold cartesian product. We give conditions for rank-one infinite measure-preserving transformations to be (weak) doubly
ergodic and for their $k$-fold cartesian product to be conservative. We also show that a (weak) doubly ergodic nonsingular group action is ergodic with isometric coefficients, and that the latter strictly implies W measurable sensitivity.
 \end{abstract}

\subjclass[2010]{Primary 37A25; Secondary 28D05}

\keywords{Ergodicity, Rational Ergodicity, Weakly Rationally Ergodic, Infinite Measure}

\maketitle 


\section{Introduction}

The weak mixing property for finite measure-preserving transformations, or actions, has several interesting and different characterizations that are all equivalent, see e.g. \cite{BeRo88}. 
In \cite{KaPa63}, Kakutani and Parry were the first to show that for
infinite measure-preserving transformations this is not the case in general. In particular they constructed, for each integer $k$, an infinite measure-preserving Markov shift $T$ such 
that the $k$-fold cartesian product $T^{(k)}= T\times\cdots\times T$ is ergodic
but the $k+1$ product $T^{(k+1)}$ is not conservative, hence not ergodic. 
In \cite{AaLiWe79}, Aaronson, Lin, and Weiss constructed an infinite measure-preserving Markov shift $T$  so that for all ergodic finite measure-preserving transformations $S$ the product $T\times S$ is ergodic but $T\times T$ is not conservative, hence not ergodic. Since that time there have been several works that have studied related examples and counterexamples, that probably can be divided into those  which have studied conditions stronger than ergodicity of the cartesian square, and those which have studied conditions weaker than ergodicity of the cartesian square; see \cite{DaSi09} for a survey of some of the work, and \cite{GlWe} for more recent results on properties weaker than ergodicity of the cartesian square. In this paper we consider a condition that is weaker than ergodicity of the cartesian square. Before stating our results we review some definitions.

We consider standard Borel measure spaces, denoted $(X,\mathcal{S},\mu)$, where $\mu$ is   a nonatomic $\sigma$-finite measure, which we assume  is infinite. We will also consider a probability measure on $(X,\mathcal S)$ which we denote by $m$ (or $m'$).  All the transformations we study are invertible and measure-preserving with respect to $\mu$ or nonsingular with respect to $m$.   A transformation $T$ is {\bf ergodic} if for every  invariant set $A$, $\mu(A)=0$ or  $\mu(X\setminus A)=0$, and   {\bf conservative} if for every measurable set  of positive measure $A$, there exists   $n\in\mathbb N$ such that $\mu(A \cap T^{-n}A)>0$. (We let  $\N$ be the set of strictly positive integers.) If $T$ is invertible and $\mu$ is nonatomic,  then it is conservative when $T$ is ergodic, see e.g. \cite[3.9.1]{Si08}.

A transformation $T$ on $(X,\mu)$ is {\bf weakly mixing} \cite{AaLiWe79} if for all ergodic  finite measure-preserving transformations $S$ on $(Y,m)$ the cartesian product $T\times S$ is ergodic .
A transformation $T$ on $(X,\mu)$ is {\bf doubly ergodic}, henceforth called  {\bf weak doubly ergodic}, or WDE, if
for every pair of measurable sets  of positive measure $A, B$, there exists a positive integer $n$  such that $\mu(A \cap T^{-n}A)>0$ and $\mu(B \cap T^{-n}A)>0$ \cite{BoFiMaSi01}. A transformation $T$ has {\bf ergodic cartesian square} if $T\times T$ is ergodic. It is clear that ergodic cartesian square implies weak doubly ergodic.  
For nonsingular group actions, ergodic cartesian square (i.e., ergodic index at least 2) has also been called doubly ergodic
(see  \cite{GlWe15} and the discussion in Section~\ref{S:EIC} below), and to avoid confusion between the two notions, in this paper we are  using  \lq\lq weak doubly ergodic"  instead of \lq\lq doubly ergodic" as in \cite{BoFiMaSi01}.   

It was shown in \cite{BoFiMaSi01} that, in infinite measure,
weak doubly ergodic does not imply ergodic cartesian square, and that while weak doubly ergodic  implies weak mixing, the converse also does not hold in infinite measure.  Thus weak doubly ergodic lies properly between weak mixing and ergodic cartesian square.

A transformation $T$ has {\bf ergodic index $k$} if $T^{(k)}$ is ergodic but $T^{(k+1)}$ is not
ergodic  \cite{KaPa63}. 
Then in the notation of Kakutani and Parry  \cite{KaPa63},  the property of ergodic cartesian square  is ergodic index at least 2. A transformation has {\bf infinite ergodic index} if all finite cartesian products are ergodic.
Similarly one defines {\bf conservative index $k$} and  {\bf infinite conservative index}.

We say that a transformation $T$ is \textbf{at least} $\boldsymbol{\rho}$-\textbf{partially rigid} for $0<\rho\leq 1$, if, for all finite measurable sets $A$, there exists a sequence $n_m \rightarrow \infty$ such that $\lim_{m\rightarrow \infty} \mu(A \cap T^{n_m}(A)) \ge \rho \mu(A)$. If $T$ is at least $\rho$-partially rigid but not at least $\ve$-partially rigid for all $\ve > \rho$, then  $T$ is called $\boldsymbol{\rho}$-\textbf{partially rigid}. The transformation is called \textbf{rigid} if $\rho$ can be taken to be $1$.

We now describe our results. In Section~\ref{S:rank1} we review rank-one transformations in infinite 
measure and the notion of descendants.  In Section~\ref{S:conservativity} we give a
condition for rank-one transformations that implies conservativity of their cartesian products. 
We note that there exist infinite measure-preserving rank-one transformations $T$ such that $T\times T$ is not conservative \cite{AdFrSi97}. In Section~\ref{S:DE} we  generalize a condition from  \cite{BoFiMaSi01} for rank-one transformations and prove that 
this condition implies  weak double ergodicity. Section~\ref{S:rigid} studies conditions for partial and full rigidity. Our main construction is in Section~\ref{S:SDE}; we construct rank-one
transformations $T$ that are weak doubly ergodic and with $T\times T$ conservative but not 
ergodic. These constructions generalize the results  in \cite{BoFiMaSi01}, where  there are  (weak) doubly 
ergodic  transformations such that
$T\times T$ is not conservative (hence not ergodic). 
Thus our examples show that when $T$ is weak doubly ergodic, even when $T\times T$ is conservative it need not be ergodic. 
In this context we mention that the original examples of Kakutani and Parry \cite{KaPa63}, as well as the examples
of Aaronson, Lin and Weiss \cite{AaLiWe79} are Markov  shifts,  and they establish that the cartesian
square is not ergodic by showing that   the cartesian square is not conservative.  
More recently, Adams and the second named author \cite{AdSi15} have constructed rigid (rank-one) transformations that are of ergodic index $k$ for each $k$.
We also show that our construction can be taken to be rigid, hence of infinite conservative index. 

In the last section we consider nonsingular group actions. Recently, 
Glasner and Weiss \cite{GlWe} have studied a property  for nonsingular group actions called ergodic with isometric coefficients, and have proved that  
   ergodic cartesian square implies this property. They also showed that there exists an
 integer infinite measure-preserving action $T$ that is ergodic with isometric coefficients 
but is such that  $T\times T$ is not conservative, hence not ergodic \cite[Proposition 7.1]{GlWe}. 
In Section~\ref{S:EIC} we show that weak double ergodicity implies 
 ergodic with isometric coefficients.  Our construction from Section~\ref{S:SDE} thus  gives 
 an infinite measure-preserving transformation $T$  that  is 
 ergodic with isometric coefficients but while $T\times T$ is conservative it is not ergodic. We also discuss notions of measurable sensitivity, and show that EIC implies but is not implied by Li-Yorke measurable sensitivity and $W$-measurably sensitivity in Proposition \ref{P:EIC-LYM}.

\textbf{Acknowledgments:} This research was supported in part by   National Science Foundation grant DMS-1347804 and the Science Center of Williams College.

\section{Rank-One Transformations}\label{S:rank1}
Our main results will be achieved through rank-one cutting and stacking constructions, which are defined as follows. A \textbf{Rokhlin column} or {\bf column} $C$ is an ordered finite collection of pairwise disjoint intervals (called the $\mathbf{levels}$ of $C$) in $\mathbb{R}$, each of the same measure. We think of the levels in a column as being stacked on top of each other, so that the $(j + 1)$-st level is directly above the $j$-th level. Every column $C = \{I_j\}$ is associated with a natural column map $T_C$ sending each point in $I_j$ to the point directly above it in $I_{j+1}$. A $\mathbf{rank}$-$\mathbf{one}$ $\mathbf{cutting}$-$\mathbf{and}$-$\mathbf{stacking}$ construction for $T$ consists of a sequence of columns $C_n$ such that:

\begin{enumerate}
	\item The first column $C_0$ is the unit interval.
	\item Each column $C_{n+1}$ is obtained from $C_n$ by cutting $C_n$ into $r_n \geq 2$ subcolumns of equal width, adding any number $s_{n,k}$ of new levels (called $\mathbf{spacers}$) above the $k$th subcolumn, $k \in \{0,\ldots,r_n - 1\}$, and stacking every subcolumn under the subcolumn to its right. In our treatment of these constructions, the spacers will be intervals drawn from $X$ that are disjoint from the levels of $C_n$ and the other spacers added to it, They will also be of the same length as the levels of the subcolumns of $C_n$ (so that $T$ is measure preserving). In this way, $C_{n+1}$ consists of $r_n$ copies of $C_n$, possibly separated by spacers.
	\item $X = \bigcup_{n \in \N} C_n$.
\end{enumerate}

\noindent
We then take $T$ to be the pointwise limit of $T_{C_n}$ as $n \rightarrow \infty$. A transformation constructed with these cutting and stacking techniques is rank-one, and we often refer to cutting and stacking transformations as rank-one. A rank-one transformation is clearly conservative ergodic.

Now given any level $I$ from $C_m$ and any subsequent column $C_n,\, n \ge m$, we define the \textbf{descendants} of $I$ in $C_n$ to be the collection of levels in $C_n$ whose disjoint union is $I$. We let the $n^\text{th}$ stage \textbf{descendant set} $D(I,n)$ contain as elements the zero-indexed heights of these levels in $C_n$. For $j \ge 0$, let $h_j$ denote the height of the topmost level in $C_j$, and set $h_{j,k} = h_j + s_{j,k}$ for $k \in \{0,\ldots,r_j - 1\}$. If $I$ is a level in $C_i$ of height $h(I)$, then the descendants of $I$ in $C_{i+1}$ are at heights $h(I)$ and $h(I) + \sum_{k=0}^\ell h_{i,k}$, $0 \le \ell < r_i - 1$. For every $n \in \N$, we set 
\[
H_n = \{0\} \cup \left\{ \sum_{k = 0}^\ell h_{n,k}\, | \, 0 \le \ell < r_n - 1 \right\},
\]
and call $H_n$ the $n^{th}$-stage \textbf{height set} of $T$. It follows that for any $I$ a level of $C_i$ and $j \ge i$, $D(I,j) = H_i + \ldots + H_{j-1}$. (For integer sets $K,L \subset \Z$, we will set $K+L = \{k + \ell : \, k\in K,\, \ell \in L\}$ as the \textbf{sum set} of $K$ and $L$. ) One can infer a number of the properties of a rank-one transformation $T$ using its height sets and resultant descendant sets, as will be seen in the following sections.

\section{Conservativity of Products}\label{S:conservativity}

We have the following equivalent condition to conservativity of products of rank-one transformations, which is proved as Proposition 4.2 in \cite{Small2014}. 

\begin{proposition} \label{P:conservprods2}
	Let $T$ be a rank-one transformation. Then the product transformation $T^{(k)}=T\times \cdots \times T$ on $X$ is conservative if and only if for every $\ve > 0$ and $i \in \N$ there is a $j> i$ such that at for at least $(1 - \ve)|D(I,j)|^k$ of the $k$-tuples $(a_0,\ldots,a_{k-1}) \in D(I,j)^k$, where $I$ is the base of column $C_i$, there exist complementary $k$-tuples $(d_0,\ldots,d_{k-1}) \in D(I,j)^k$ satisfying $a_0 - d_0 = a_\ell - d_\ell \neq 0 $ for $\ell = 1,\ldots,k-1$. 
	\end{proposition}

This leads to the following theorem on the conservativity of products of rank-one transformations.

\begin{theorem}
Let $r_n = |H_n|$ and let $k \in \N$. If 
\[\prod_{n=0}^\infty \left(1-\frac{1}{r_n^{k-1}}\right) = 0,\]
then  the product transformation $\left(T^\ell\right)^{(k)}$ is conservative for any $\ell$. \label{T:conscuts}
\end{theorem}

\begin{proof}
 Let $I$ be the base of column $C_i$ for any fixed $i \in \N$. Fix any $\ve>0$. For any $j \ge i$, the descendant set $D(I,j)$ is given by $D(I,j) = H_i + H_{i+1} + \ldots + H_{j-1}.$
Let $(a_0,\ldots,a_{k-1})$ be a $k$-tuple in $D(I,j)^k$. For each $a_\ell \in D(I,j)$ with $\ell \in \{0,\ldots,k-1\}$, we can write $a_\ell = \sum_{ m = i}^{j-1} a_{\ell,m}$, where $a_{\ell,m} \in H_m$. Let $\rho_j$ denote the fraction of $k$-tuples in $D(I,j)^k$ of the form $(a_0,\ldots,a_{k-1})$ which do not have $a_{0,p} = a_{1,p} = \dots = a_{k-1,p}$ for some $p \in \{i,\ldots,j-1\}$. Then
\[\rho_{j+1} = \left(1-\frac{|H_j|}{|H_j|^k} \right)\rho_j =\left(1-\frac{1}{r_j^{k-1}}\right) \rho_j.\]
If $T$ satisfies the stated condition then
\[\rho_j \le \prod_{m = i}^{j-1} \left(1-\frac{1}{r_m^{k-1}}\right) \rightarrow 0,\]
so we can find a $j'\in \N$ such that at least some fraction $1-\ve$ of the $k$-tuples in $D(I,j')^k$ of the form $(a_0,\ldots,a_{k-1})$ have $a_{0,p} = a_{1,p} = \ldots = a_{k-1,p}$ for some $p \in \{i,\ldots,j'-1\}$. Consider such a $k$-tuple. Let $\gamma \in H_p$ be any other element in $H_p$ (i.e., $\gamma \ne a_{0,p}$), and set 
\[d_\ell = \gamma+ \sum_{\substack{m = i\\ m \ne p}}^{j'-1} a_{\ell,m} ,\]
which is also an element of $D(I,j')$.Because this holds for an arbitrarily high proportion of the tuples $(a_0,\ldots,a_{k-1}) \in D(I,j')^k$, Proposition \ref{P:conservprods2} implies that $T^{(k)}$ is conservative. 

It is known that composition powers of a conservative transformations are conservative (see e.g.\ \cite[Corollary 1.1.4]{Aa97}). Hence, $\left(T^\ell\right)^{(k)}$ must be conservative for any $\ell$. 
\end{proof}

\noindent\textbf{Remark}: It can easily be seen that the converse of Theorem \ref{T:conscuts} does not hold. For example, by letting the height sets of $T$ be arithmetic progressions with very large and quickly increasing lengths, we can obtain a transformation which is rigid.

\section{Double Ergodicity Preliminaries}
\label{S:DE}

To establish the non-ergodicity of cartesian squares of rank-one transformations, we require the following lemma.

\begin{lemma} \label{L:kproduct} For a rank-one cutting and stacking transformation $T$,  if $T\times T$ is conservative ergodic, then for every $i\in \N$, $\ve>0$, $b \in \{0,\ldots,h_i-1\}$, there is a natural number $j>i$ such that for at least $(1-\ve)|D(I,j)|^2$ pairs of descendants of the base $I$ of column $C_i$ of the form $(a,a') \in D(I,j)^2$, we have corresponding pairs $(d,d') \in D(I,j)^k$ such that $a-d = a'-d' - b$.
\end{lemma}
\begin{proof}
This is a straightforward application of Lemma 2.4 from \cite{Small2014} with $\alpha = (1,1)$.
\end{proof}

Say that a set $C$ is $\mathbf{(1-\boldsymbol{\ve})}$-\textbf{full} of $D$ if $\mu(D\cap C) \ge (1-\ve) \mu(C)$. The following lemma is standard:

\begin{lemma} \label{L:levelsa} Let $T$ be a rank-one cutting and stacking transformation. Given $\ve>0$ and any sets $A,B\subset X$, both of positive measure, there exist intervals $I$ and $J$ in some column $C_n$ of $T$, with $I$ above $J$, such that $I$ is $(1-\ve)$-full of $A$ and $J$ is $(1-\ve)$-full of $B$. 
\end{lemma}

\noindent The following concerns weak double ergodicity and is Lemma 5.3 from \cite{Small2000}. 

\begin{lemma} \label{L:levelsb} Let $T$ be a rank-one transformation. Let $A,B\subset X$ be sets of positive measure, and choose any levels $I,J\subset C_n$ such that $\mu(I\cap A)+\mu(J\cap B)>\delta\mu(I)$, with $I$ a distance $\ell\ge 0$ above $J$. If we cut $I$ and $J$ into $r_n$ equal pieces $I_0,\ldots,I_{r_n-1}$ and $J_0,\ldots,J_{r_n-1}$, respectively (numbered from left to right), then there is some $k$ such that 
\[\mu(I_k\cap A)+\mu(J_k\cap B)>\delta \mu(I_k),\]
and $I_k$ will be $\ell$ above $J_k$ in $C_{n+1}$.  
\end{lemma}

\noindent A \textbf{staircase transformation} is one for which we cut every column $C_n$ into $r_n$ pieces, and place $i$ spacers over its $i^\text{th}$ ($0$-indexed) subcolumn, so that its height set elements are sums of multiples of $h_n$ with triangular numbers. They became of interest when Adams \cite{Ad98} proved that the classical finite measure-preserving staircase transformations ($r_n=n$) is mixing.  In   \cite{Small2000}, \textbf{tower staircases} were defined as a staircase but with no restriction on the number of spacers in the last subcolumn and  it was shown that      all staircase transformations are (weakly) doubly ergodic \cite[Theorem 2]{Small2000}. We show in Proposition~\ref{P:hardstairs}   that this holds more generally for transformations that contain infinitely many height sets with a partial staircase configuration. 

A \textbf{high staircase} transformation, as defined in \cite{DaRy11}, is a modified staircase transformation in which we take $r_n \rightarrow \infty$ as $n \rightarrow \infty$, and we place $i + z_n$ spacers over the $i^\text{th}$ subcolumn of $C_n$, where $(z_n)_{n=1}^\infty$ is a sequence of positive integers. A high staircase is a tower staircase. Corollary \ref{C:highstaircaseDE} shows that high staircase transformations are all weak doubly ergodic, giving another proof of the result from
\cite{Small2000}. However, in Corollary \ref{C:notpwm}, we show that not all high staircase transformations are power weakly mixing (indeed, not all such transformations have conservative cartesian square), providing a counterexample to a conjecture made in \cite{DaRy11}, wherein high staircases that are power weakly mixing are constructed.

A rank-one transformation is an  \textbf{arithmetic rank-one} transformation if 
there exists an infinite sequence $(n_i)_{i\in \N}$ indexing nonnegative integers $a_{n_i}$ and height sets $H_{n_i}$ which contain subsets of the form:
\begin{align*}
S_{n_i}=\Bigg\{a_{n_i},\, a_{n_i}+h_{n_i}+k_{n_i} + 1,\, a_{n_i}+2h_{n_i}+2k_{n_i}+3,\ldots,&\\
a_{n_i}+(s_{n_i}-1)h_{n_i}+\sum^{s_{n_i}-1}_{q=0}\left(k_{n_i}+q\right)&\Bigg\},
\end{align*}
such that for all $i\in \N$, $k_{n_i}\ge 1$, $\frac{s_{n_i}}{r_{n_i}}>\tau$ for some fixed $\tau>0$, and $r_{n_i}$ is unbounded,
and $r_n\ge 2$ for all $n$.
 If in addition $(n_i)_{i \in \N} = i -1$, $a_{n_i} = 0$, and $H_{n_i} = S_{n_i}$ for all $i$ we say it is a \textbf{ strongly arithmetic   rank-one} transformation.

\begin{proposition} If $T$ is an arithmetic rank-one transformation, then it is weak doubly ergodic. \label{P:hardstairs}
\end{proposition}

\begin{proof}
Let $\ve<\frac{1}{4}$ and set $\nu$ positive with $\nu<\frac{\ve\tau}{2}$. Let $A$ and $B$ be sets of positive measure in $X$. By Lemma \ref{L:levelsa}, we can find levels $L$ and $M$ in some column $C_N$ with $L=T^\ell M$ and $L$ and $M$ both $\left(1-\frac{\nu}{2}\right)$-full of $A$ and $B$, respectively. By supposition, there exists an $i'\in \N$ such that $s_{n_{i'}}>\frac{5\ell+5}{1-4\ve}$. For brevity, set $n=n_{i'}$. Applying Lemma \ref{L:levelsb} $n-N$ times, we can find levels $I$ and $J$ in $C_n$ such that $\mu(I\cap A)+\mu(J\cap B)>(2-\nu)\mu(I)$, and $I=T^\ell(J)$ (that is, $I$ is $\ell$ above $J$). Obviously, this implies that 
\begin{align}
\mu(I\cap A)>(1-\nu)\mu(I) \qquad \label{E:cont1}
\mu(J\cap B)>(1-\nu)\mu(J). 
\end{align}
Now let $I_0,\ldots,I_{s_n-1}$ denote from left to right the descendants of $I$ in column $C_{n+1}$. Define $J_0,\ldots,J_{s_n-1}$ similarly. By \eqref{E:cont1}, fewer than $2\nu r_n$ of the descendants of $I$ are less than half-full of $A$. But by selection,
\[\frac{2\nu r_n}{s_n}<\frac{2\ve \tau r_n}{2s_n}<\ve.\]
So more $(1-\ve)s_n$ of the subintervals in $I_0,\ldots,I_{s_n-1}$ are more than half-filled with $A$ and similarly for $J_0,\ldots,J_{s_n-1}$ and $B$. Note that for any $j,\, 1\le j \le s_n-\ell-1$, we have that 
\begin{align*}
&T^{h_n+k_n+j}I_{j-1}=I_{j}\\
&T^{h_n+k_n+j}I_{j+\ell - 1}=T^{-\ell}(I_{j+\ell})=J_{j+\ell}.
\end{align*}
Let $K_1,K_2,L_1,L_2$ be subsets of $\{0,\ldots,s_n-\ell-2\}$ satisfying
\begin{align*}
K_1&=\left\{j\in \{1,\ldots,s_n-\ell-1\}:\, \mu\left(I_{j-1}\cap A\right)>\frac{1}{2}\mu\left(I_{j-1}\right)\right\},\\
K_2&=\left\{j\in \{1,\ldots,s_n-\ell-1\}:\, \mu\left(I_{j}\cap A\right)>\frac{1}{2}\mu\left(I_{j}\right)\right\},\\
L_1&=\left\{j\in \{1,\ldots,s_n-\ell-1\}:\, \mu\left(I_{j+\ell-1}\cap A\right)>\frac{1}{2}\mu\left(I_{j+\ell-1}\right)\right\}, \text{ and}\\
L_2&= \left\{j\in \{1,\ldots,s_n-\ell-1\}:\, \mu(J_{j+\ell})\cap B>\frac{1}{2}\mu\left(J_{j+\ell}\right)\right\}.
\end{align*}
We have $|K_1|,|K_2|,|L_1|,|L_2|>s_n-\ell-1-\ve s_n$. 
Thus, 
\begin{align*}
\left|\{1,\ldots,s_n-1\} \setminus \Big(K_1 \cap K_2 \cap L_1 \cap L_2 \Big)\right| \le 4(\ell+1+\ve s_n)\\
=4\ell+4+4\ve s_n<s_n-\ell-1,
\end{align*}
by selection of $n$ and $s_n$. So there exists some $j^*\in K_1 \cap K_2 \cap L_1 \cap L_2$ with
\begin{align*}
&\mu\left(T^{h_n+k_n+j^*} (A\cap I_{j^* -1})\cap (A\cap I_{j^*})\right)>0\\
&\mu\left(T^{h_n+k_n+j^*} (A\cap I_{j^*+\ell -1})\cap (B\cap J_{j^*+\ell})\right)>0.
\end{align*}
So $T$ is weak doubly ergodic. 
\end{proof}

\noindent As an easy consequence, we have:

\begin{corollary} [\cite{Small2000}]\label{C:highstaircaseDE}
All high staircase transformations are weak doubly ergodic.
\end{corollary}

\noindent It is clear that the height sets of staircase transformations make extensive use of triangular numbers. The following lemma is an easy observation which we will exploit later.

\begin{lemma} \label{L:tone} Fix integers $a,b\in \N$ and $c\in \Z$ such that $|c|<\min \{a,b\}$. Let $x_i=\frac{i(i+1)}{2}$ ($x_i$ is the $i^\text{th}$ triangular number). If $a\ne b$, then $|x_{a+c}-x_a-(x_{b+c}-x_b)|\ge c$.
\end{lemma}

\begin{lemma}\label{L:two} Let $k$ be a fixed integer and for an $r\in \N$, define $H(r)=\{1,\ldots,r\}$. Fix an $m\in \N$, and let $\ve \in \R$ with $\ve>0$. Let $D(r,m)\subset H(r)\times H(r)$ be the set of pairs $(a,d)\in H(r)\times H(r)$ such that $|a-d|< m$. Then 
\[\lim_{r\rightarrow \infty} \frac{|D(r,m)|}{|H(r)^2|}=0.\]
\end{lemma}
\begin{proof}
Let $r>m$. Let $A_1=\big\{1,2,\ldots,m\big\}\subset H(r)$, $A_2=\big\{2,3,\ldots,m+1\big\}$, and construct sets similarly until we have $A_{r-m}=\big\{r-m,\ldots,r\big\}$. The number of pairs in $A_1^2$ is $|A_1|^2= m^2$. The number of pairs in $A_2^2$ is also $m^2$, but $(m-1)^2$ of these pairs are also in $A_1^2$ (specifically, those that occur in $\{2,m\}^2$. So $|A_1^2\cup A_2^2|=2m^2-(m-1)^2$. Continuing this process for $A_3,\ldots,A_{r-m}$, we see that 
\[\left| \bigcup^{r-m}_{i=1} A_i^2\right|=(r-m)m^2-(r-m-1)(m-1)^2= 2mr-m^2-m-r+1.\]
Of course, if $(a,d)\in D(r,m)$, then $(a,d)\in A_i$ for some $i$. Also, if $(a,d)\in A_i$ for some $i$, then clearly $(a,d)\in D(r,m)$. Hence 
\begin{align}
&\frac{|D(r,m)|}{|H(r)^2|} = \frac{2mr-m^2-m-r+1}{r^2}, \text{ and therefore} \label{E:pairsinequality} \\
&\lim_{r\rightarrow \infty} \frac{|D(r,m)|}{|H(r)|^2}=0.  \nonumber
\end{align}
\end{proof}

\section{Partial Rigidity and Rigidity} \label{S:rigid}

In this section, we will see that rigidity can be determined in the notation of height sets, with an application to the transformation discussed in Theorem \ref{T:dergodic}. We begin with the following proposition, which gives us a useful way of constructing rank-one transformations with partial and full rigidity using only the height sets. As standard notation, for $k \in \N$, let $[k] = \{0,\ldots , k \}$.

\begin{proposition} Let $T$ be a rank-one transformation defined with its height sets $(H_n)_{n\in \N}$. If there is a sequence $n_m \rightarrow \infty$ and a corresponding sequence $a_{n_m}$ of positive integers such that 
\[\lim_{m\rightarrow \infty} \frac{\left|(a_{n_m} + H_{n_m}) \cap H_{n_m} \right|}{\left| H_{n_m}\right|} = \rho,\]
then $T$ is at least $\rho$-partially rigid. \label{P:partialrigid}
\end{proposition}

\begin{proof}

First, for any $n \in \N$, let $S_n\subset \{0,\ldots,h_{n-1}\}$ be any subset of $[h_n]$. Let $I_n$ denote the base of $C_n$ and $I_{n+1}$ the base of $C_{n+1}$. Recall that 
\[\bigcup_{j \in S_n} T^j I_n = \bigcup_{j \in H_n + S_n} T^j I_{n+1}.\]
Suppose that for some natural number $k$, $\left|(k+ H_n ) \cap H_n \right| = \gamma \left| H_n \right|$, where $\gamma \in (0,1)$. It follows that 
\[\left| (k+ H_n + S_n ) \cap (H_n + S_n)\right| \ge \gamma \left| H_n + S_n\right|,\]
implying 
\begin{align}
\mu \Bigg( \Bigg(\bigcup_{j \in S_n} T^j I_n\Bigg) &\cap T^k\Bigg(\bigcup_{j \in S_n} T^j I_n\Bigg)\Bigg) \nonumber\\
& \ge \mu\left(\bigcup_{j \in \left(k+H_n + S_n\right)\cap \left( H_n + S_n \right)} T^j I_{n+1} \right) \nonumber \\
& \ge \gamma \mu \left( \bigcup_{j \in H_n + S_n} T^j I_{n+1}\right) = \gamma \mu\left(\bigcup_{j \in S_n} T^j I_n\right). \label{E:htsetrigid}
\end{align}

Now fix any finite measure set $A$ and $\ve>0$ with $\ve < \mu(A)$. Set an $N \in \N$ such that for every $n\ge N$ there exists some set of levels $D^*(A,n) \subset [h_n]$ and corresponding set $B = \bigcup_{j \in D^*(A,n)}T^j (I_n)$, where $I_n$ is the base of $C_n$, such that $\mu(A\, \Delta \, B) < \frac{\ve}{6}$. So $A$ is closely approximated by the set of levels $B$. By the assumed conditions on $T$, we can also fix $N$ high enough such that $\left| (a_{n_m} + H_{n_m}) \cap H_{n_m} \right| > \left(\rho- \frac{\ve}{2\mu(A)}\right)$ whenever $n_m \ge N$. For brevity, let $n=n_m \ge N$. By application of \eqref{E:htsetrigid}, 
\[\mu\left( B \cap T^{a_n} B \right)  >  \left(\rho- \frac{\ve}{2\mu(A)}\right) \mu(B).\]
Noting that $\mu(B)\ge \mu(A)-\mu(A\,  \Delta B) > \mu(A) - \frac{\ve}{6}$, this allows us to write that 
\begin{align*}
\mu\left(A \cap T^{a_n} A\right) &\ge \mu\left( (A\cap B) \cap T^{a_n} (A\cap B) \right)\\
& \ge \mu\left( B \cap T^{a_n} B \right) - 2\mu(B\setminus A)\\
& \ge \left(\rho- \frac{\ve}{2\mu(A)}\right) \mu(B) - \frac{2 \ve}{6} \\
& >\left(\rho- \frac{\ve}{2\mu(A)}\right) \left( \mu(A) - \frac{\ve}{6}\right) - \frac{2\ve}{6} \\
& = \rho \mu(A) - \ve.
\end{align*}
Note that we necessarily have $a_{n_m} \ge \min\left((H_{n_m}-H_{n_m})\cap \N\right)$, or else $\left(a_{n_m} + H_{n_m} \right)\cap H_{n_m} = \emptyset$ when $a_{n_m}\ne 0$. But 
\[
\lim_{m\rightarrow \infty} \min\left((H_{n_m}-H_{n_m})\cap \N \right) = \infty,
\]
 so $\lim_{m\rightarrow \infty} a_{n_m} = \infty$. Hence, the sequence $(a_{n_m})$ forms a rigidity sequence for any $A$. 
\end{proof}

\section{A Strictly weak doubly ergodic, Rigidity-Free Transformation $T$}
\label{S:SDE}
 
We now construct a weak doubly ergodic transformations with $T\times T$ conservative but not ergodic. First, we define $T$ by its height set sequence. For even $n$, let $H_n = \{0,g_n\}$, where $g_n$ is chosen to satisfy $g_n > 2 \max D(I,n)+1$. By Lemma \ref{L:two}, we can choose the number of cuts to be employed in $(n+1)^\text{th}$ height set $r_{n+1}$ so high that at least $1-\frac{1}{4(n+1)^2}$ of the pairs $\left(i,j\right)\in \{0,\ldots,r_{n+1}-1\}^2$ have $|i-j|>2\max D(I,n+1)+1 = 2 \max D(I,n) + 2g_n + 1$. Place enough spacers on the rightmost subcolumn of $C_n$ such that 
\begin{equation} \label{E:hineq}
h_{n+1} > 2 \sum_{i=1}^{r_{n+1} - 1} i + 2\max D(I,n+1) + 1. 
\end{equation}

For odd $n$, Set 
\[H_{n} = \left\{ 0, \, h_{n} +1,\, 2h_{n}+3,\, \ldots,\, (r_{n}-1) h_{n} + \sum_{i=1}^{r_n-1} i \right\}.\]
For notational ease, we may write 
\[H_n = \big\{ \beta_0,\ldots,\beta_{r_n-1} \big\},\]
where $\beta_i = i h_n + x_i - 1$, and $x_i$ is the $i^\text{th}$ triangular number.

\begin{lemma}\label{L:tqnotergodic}
$T\times T$ is not ergodic.
\end{lemma}

\begin{proof}
Fix an odd $n \in \N$. Let $(\beta_i,\beta_j)$ be one of the pairs in $H_n$ such that $|i - j| > 2\max D(I,n) + 1$. Let $\beta_h$ and $\beta_\ell$ be two other elements of $D(I,n)$---we claim that if $\beta_i - \beta_j \ne \beta_h + \beta_\ell$, then the following inequality always holds:
\[\left|\beta_i-\beta_j-\beta_h+\beta_\ell\right|>2\max D(I,n)+1.\]
Suppose first that we had $h-\ell\ne i-j$. Then we should have
\begin{align*}
\left|\beta_i-\beta_j-\beta_h+\beta_\ell\right|&=\left|(i-j-h+\ell)h_n+x_{i-1}-x_{j-1}-x_{h-1}+x_{\ell-1}\right|\\
&\ge \left|h_n-2x_{r_n-1}\right|\\
&=2\max D(I,n)+1.
\end{align*}
where the last line follows from \eqref{E:hineq}. On the other hand, if $h-\ell=i-j$, then the situation is the one encountered in Lemma \ref{L:two}, with $c=j-i$ and $a=i-1$. That is, if $\beta_i-\beta_j \ne \beta_h-\beta_\ell$ (implying that $i\ne h$) we should have 
\begin{align*}
\left|\beta_i-\beta_j-\beta_h+\beta_\ell\right|
&=\left|x_{i-1+j-i}-x_{i-1}-x_{j-i+h-1}+x_{h-1}\right|\\
&=\left|i-j\right|\left| i-h\right|
\ge |i-j|>2\max D(I,n)+1.
\end{align*}
So, letting $(\beta_i,\beta_j)$ be one of the pairs such that $|i-j|>2\max D(I,n)+1$, there is no pair $(h,\ell)\in \big\{0,\ldots,r_n-1\big\}^2$ such that $\left|\beta_i-\beta_j-\beta_h+\beta_\ell\right|\ne 0$ but $\left|\beta_i-\beta_j-\beta_h+\beta_\ell\right|\le 2\max D(I,n)+1$. Let $K_n$ denote the subset of $H_n^2$ consisting of precisely the pairs $(\beta_i,\beta_j)$ where $|i-j|$ exceeds $2\max D(I,n)+1$. By construction, we have that $\frac{|K_n|}{|H_n|^2}>1-\frac{1}{4n^2}$.

For any $n\in \N$ and any pair $(a,a')\in D(I,n)$, we use the sum decomposition $a=\sum_{i=0}^{n-1} a_i$ and $a'=\sum_{i=0}^{n-1} a_i'$, where $a_i,a_i'\in H_i$ for every $i,\, 0\le i \le n-1$. Let $F_n$ denote the subset of $D(I,n)^2$ consisting of pairs $(a,a')\in D(I,n)^2$ such that for every odd $i\in \{1,\ldots,n-1\}$, $(a_i,a_i')\in K_i\subset H_i^2$. Then
\begin{align*}
\frac{|F_n|}{|D(I,n)|^2}&=\prod_{i=0}^{\lfloor \frac{n-1}{2}\rfloor} \frac{\left|K_{2i+1}\right|}{\left|H_{2i+1}\right|}\ge \prod_{i=0}^{\lfloor \frac{n-1}{2}\rfloor} \left(1- \frac{1}{4(2i+1)^2}\right)>\frac{1}{2}.
\end{align*}

Finally, let $(a,a')\in F_n$. Let $(d,d')\in D(I,n)$, where $a-a'\ne d-d'$. We will show $a-a'\neq d-d'+1$. Let $k$ be the highest integer such that $a_k-a_k'\ne d_k-d_k'$; such an integer clearly must exist between $0$ and $n-1$. By selection of $a$ and $a'$, if $k$ is odd, then $\left|a_k-a_k'-d_k+d_k'\right|>2\max D(I,k)+1$. Alternatively, if $k$ is even, then $a_k-a_k'$ and $d_k-d_k'$ must differ by at least $g_k > 2\max D(I,k)+1$. Hence, 
\begin{align*}
\left|a-a'-d+d'\right|>2\max D(I,k)+1-2\max D(I,k)=1
\end{align*}
So for any $n$, over half of the pairs $(a,a')$ in $D(I,n)^2$ have no complementary pair $(d,d')$ satisfying $a-a'=d-d'+1$. Thus, by Lemma \ref{L:kproduct}, $T_q\times T_q$ is not ergodic. 
\end{proof}

By a similar proof, the following lemma holds for strongly arithmetic rank-one transformations (where every height set has the staircase form). 

\begin{lemma}\label{L:arithmeticnotcons}
Let $T$ be a strongly arithmetic rank-one transformation with cutting sequence $(r_n)_{n =0}^\infty$, and let $K_n$ denote the subset of $H_n^k$ consisting of the pairs $(\beta_{i_1},\ldots,\beta_{i_k})$ where $\max_{\substack{1\le m,\ell \le k, \\m \ne \ell}} |i_m - i_\ell| > 2\max D(I,n)$ ($|K_n|$ is easy to calculate, as in Lemma \ref{L:two}). If $\prod_{j = 0}^\infty \left( 1- \frac{|K_j|}{r_j^k} \right) > 0,$ then $T^{(k)}$ is not conservative. 
\end{lemma}

\begin{proof}
By Proposition \ref{P:conservprods2}, it suffices to show that the fraction of $k$-tuples $(a_1,\ldots,a_k) \in D(I,n)^k$ without complementary pairs $(d_1,\ldots,d_k)$ satisfying $a_1 - d_1 = a_i - d_i,\, i = 2,\ldots,k$ is bounded below by some $\ve > 0$ for every $n$. As in the proof of Lemma \ref{L:tqnotergodic}, Lemma \ref{L:tone} guarantees that any $k$-tuple $(a_1,\ldots,a_k)$ with $(a_{1j},\ldots,a_{kj}) \in K_j \subset H_n$ for all $j = 0,\ldots, n-1$ will not have a complementary pair $(d_1,\ldots,d_k)$ with this property. But the fraction of such pairs is bounded bounded below by $\prod_{j = 0}^\infty \left( 1- \frac{|K_j|}{r_j^k} \right) > 0$.
\end{proof}

This has clear implications for high staircase transformations, which were conjectured to be power weakly mixing by the authors in \cite{DaRy11}. Specifically, by Lemma \ref{L:arithmeticnotcons}, when $(r_n)$ increases sufficiently fast, we can guarantee that $T\times T$ is not conservative for a strongly arithmetic rank-one $T$:

\begin{corollary} \label{C:notpwm} There exist high staircase transformations which do not have conservative cartesian squares.
\end{corollary}

\noindent However, it is possible that all strongly arithmetic rank-ones satisfying the \textit{restricted growth} condition used in  \cite{DaRy11} (i.e. $\lim_{n \rightarrow \infty} \frac{r_n^2}{r_0 r_1 \dots r_{n-1}} = 0$) are power weakly mixing, as examples which use Lemma \ref{L:arithmeticnotcons} can easily be seen to not satisfy this condition.

We now return to the main transformation of this section. We prove that there exist weak doubly ergodic transformations with conservative though not ergodic cartesian square. We note examples of tower rank-one transformations that are
weak doubly ergodic but with non-conservative cartesian square were constructed in \cite{BoFiMaSi01}. Also, as Aaronson has mentioned to the authors \cite{Aa13p}, for  $0<t<1$, if   $u_n=(n+1)^{-t} $, then $\{u_n\}$ is a Kaluza sequence, thus there  exists an
invertible, rational weak mixing Markov shift $T$ with a state 0 so that
$p_{0,0}^{(n)}=u_n.$ For $t>1/2$, $T\times T$ is not is not conservative. By \cite{DGPS15}, rational weak mixing
implies weak doubly ergodic, so this give another example of a double ergodic transformation with non-conservative cartesian square. Our examples below can also be chosen to be rigid (Theorem~\ref{T:dergodic2}). We note that rigid rank-one transformations with cartesian square were constructed in \cite{BaYa15}, and 
rigid rank-one transformations with  infinite ergodic index were constructed in \cite{AdSi15}.

\begin{theorem} There exists a weak doubly ergodic $T$ which is partially rigid but with $T\times T$ non-ergodic.  \label{T:dergodic}
\end{theorem}

\begin{proof}
$T$ is weak doubly ergodic by Proposition \ref{P:hardstairs} and at least $\frac{1}{2}$-partially rigid by application of Proposition \ref{P:partialrigid} to the even height sets $H_{2n},\, n \in \N$, with $a_{2n} = \gamma_{2n}$. In addition, $T \times T$ is not ergodic by Lemma \ref{L:tqnotergodic}.  
\end{proof}

Theorem \ref{T:dergodic} can be extended to fully rigid transformations, as we show next.

\begin{theorem}There exists a transformation $T$ which is rigid and weak doubly ergodic, but with $T\times T$ not ergodic. \label{T:dergodic2}
\end{theorem}
\begin{proof}
For $n$ even, set $H_n = \big\{0,\gamma_n,2\gamma_n,\ldots,(n-1)\gamma_n\big\}$, where $\gamma_n= 2h_n$. Thus, the height set sequence contains a subsequence of arithmetic progressions of increasing length, so Proposition \ref{P:partialrigid} implies rigidity. 

As was the case in Theorem \ref{T:dergodic}, for $n$ even, add enough spacers on the rightmost subcolumn of column $C_n$ for $n$ even in order for 
\[h_{n+1}>\left\{2\sum_{i=1}^{r_n-2} i + 2\max D(I,n+1)+1\right\},\] where we have selected a number of cuts $r_{n+1}$ satisfying equation (\ref{E:distanceinequality}) for $m = 2\max D(I,n+1)+1$. Then, set $H_{n+1}$ equal to the staircase height set with $r_{n+1}$ subcolumns.
The argument is finished by methods already employed. 
\end{proof}

\section{Weaker Notions of Mixing}\label{S:EIC}

\subsection{Weak Double Ergodicity Implies EIC} 

In \cite{GlWe15}, Glasner and Weiss studied, for nonsingular actions of locally compact second countable groups, several conditions that are weaker that ergodicity of the cartesian square and stronger that ergodicity of the action. In this section we show that weak double ergodicity is stronger than ergodicity with isometric coefficients (EIC). This is the only section where we consider nonsingular group actions.

Let $G$ be a locally compact second countable topological group. A (Borel) $G$-action on $(X,m)$ is {\bf weak doubly ergodic} if for all sets $A, B\subset X$ of positive measure there exits $g\in G$ such that $m(T_g A\cap A)>0$ and $m(T_g A\cap B)>0$. As mentioned in \cite{GlWe15}, double ergodicity has also been used to mean ergodicity of the diagonal action on the cartesian square (see \cite{GlWe15} and references therein); this is different from double ergodicity as defined in \cite{BoFiMaSi01} and in order to differentiate the notions we are using weak double ergodicity for the notion in \cite{BoFiMaSi01}.

The following lemma is a straightforward generalization of Proposition 2.1 from \cite{BoFiMaSi01}, which handles the case of integer actions. The proof is essentially the same. 

\begin{lemma} \label{L:kconserg}
Let $\{T_g\}_{g \in G}$ be a group action on $\sigma$-finite space $X$; then $T$ is weak doubly ergodic if and only if for every $k\in\N$ and positive measure sets $A_i,B_i,i=1\ldots,k$ there exists $g\in G$ such that $\mu(T_g A_i \cap B_i)>0$ for all $i=1,\ldots k$.
\end{lemma}

A nonsingular action $\{T_g\}_{g\in G}$ is said to be {\bf ergodic with isometric coefficients (EIC)},
see \cite{GlWe15} and references therein, if every factor map (i.e., Borel equivariant map) 
$\phi: X \to Z$ where $(Z,d)$ is a separable metric space and where the factor action acts by isometries is constant a.e. In \cite{GlWe15} the authors show that ergodic cartesian square implies EIC, and that EIC implies weak mixing (which is defined analogously to weak mixing for the integer action case--they also consider other notions which we do not address in this paper). In this section we first observe that ergodic cartesian square clearly implies weak double ergodicity and in Proposition~\ref{P:WDE-EIC} prove that weak double ergodicity implies EIC. In \cite{GlWe15} it is also shown that there exist infinite measure-preserving $\Z$ actions $T$  that are EIC but not ergodic cartesian square. The proof that $T$ is not ergodic cartesian square is obtained by showing its cartesian square is not conservative. 

In \cite{BoFiMaSi01}, the authors construct infinite measure-preserving rank one transformations $S$ such that $S$ is (weak) doubly ergodic, hence EIC by \ref{P:WDE-EIC}, but $S\times S$ is not conservative, hence not ergodic. The examples we construct here give integer actions that are EIC (as they are weakly double ergodic) with conservative but not ergodic cartesian square.  The equivalence of EIC with weak double ergodicity is left open.

\begin{proposition}\label{P:WDE-EIC}  Let $\{T_g\}_{g\in G}$ be a nonsingular properly ergodic action on $(X,m)$. If $\{T_g\}$ is weakly doubly 
ergodic, then it is ergodic with isometric coefficients (EIC). 
\end{proposition}

\begin{proof}
Let $\phi: (X,T_g,m)\to (Z,S_g,\phi*m)$ be a factor map, where we can assume the metric space $(Z,d)$ is separable and $S_g$ is an isometry 
for each $g\in G$.  Set $m'=\phi*m$. Let $x, y$ be points in the support of the measure $m'$, which we may assume are distinct. Let $r=d(x,y)>0$. As all (positive radius) balls centered at $x$ and $y$ have positive measure, and factors are also weakly doubly 
ergodic, there exists $g\in G$ such that 
$m'(S_g B(x,r/4)\cap B(x,r/4)>0$ and $m'(S_g B(x,r/4)\cap B(y,r/4)>0$. As $S_g$ is an isometry,
$S_g(B(x,r/4)$ is a ball also of radius $r/4$, contradicting the triangle inequality. Therefore the factor is trivial and the action $\{T_g\}$ is EIC.
\end{proof}

\subsection{Measurable Sensitivity}

A related notion to EIC that we discuss now only for transformations is W-measurable sensitivity, which is defined in \cite{GrIo12}. A transformation is said to be \textbf{W-measurably sensitive} if for all $\mu$-compatible metrics $d$ (that is, all metrics which satisfy $\ve>0 \implies \mu(B(x,\ve))> 0$ for a.e.\ $x$), there is a $\delta> 0$ such that for almost every $x \in X$, 
\begin{align} \label{E:wsens}
\limsup_{n\rightarrow \infty} d(T^n x,T^n y) >\delta
\end{align}
for almost every $y \in X$. 
A classification result in Theorem 1, \cite{GrIo12} shows that all conservative ergodic nonsingular transformations $T$ are W-measurably sensitive or isomorphic mod $\mu$ to an invertible minimal uniformly rigid isometry on a Polish space. A transformation is called $\mathbf{L}_{\boldsymbol{\infty}}$-\textbf{weak mixing} if the space $L_\infty(X,\mu)$ has no nontrivial invariant subspaces of finite dimension. We show that $W$-measurable sensitivity need not imply $L_\infty$ weak mixing, which is one of the weakest notions of mixing in the infinite measure case (in integer actions, this is equivalent to weak mixing: see Remark 0.2 of \cite{GlWe15}).  Theorem 1 of \cite{GlWe15} gives the following useful chain of implications, into which we introduce WDE: 
\begin{align*}
T\times T \text{ ergodic} \implies \text{WDE} \implies \text{EIC} \implies \text{WM} \implies L_\infty \text{-WM} 
\end{align*}

\begin{proposition}\label{P:EIC-LYM}   Let $T$ be a conservative ergodic nonsingular transformation on $(X,m)$. If $T$ is  EIC then it is $W$-measurably sensitive. However, Li-Yorke M-Sensitivity does not imply $L_\infty$-weak mixing.  
\end{proposition}

\begin{proof}
If $T$ is EIC but not $W$-measurably sensitive, then Theorem 1 in \cite{GrIo12} shows that it must be isomorphic mod $\mu$ to an isometry on a Polish space, so taking $\phi$ as the isomorphism shows that $T$ is not EIC. The resulting metric space is separable by $\mu$-compatibility of the metric (Lemma 5.3 of \cite{GrIo12}), and the fact that any $\mu$-compatible psuedometric is separable (Proposition 2.1, \cite{GrIo12}).

To show that the converse does not hold, consider a transformation for which $T^2 \times T^2$ is ergodic nonsingular on a measure space $X$. Let $Y = \{0,1\} \times X$ and the transformation $S$ on $Y$ be defined by $S(t,x) = (t+1\mod{2},T(x))$. Letting $X_0 = \{0\} \times X$, $X_1 = \{1\} \times X$, then clearly $S^2 \times S^2$ is ergodic on each element of the partition $\{X_0 \times X_0, X_1 \times X_1, X_0 \times X_1,X_1 \times X_0\}$. The transformation $S^2 \times S^2$ is also closed on each of these sets. Now there exists a $\delta > 0$ such that for any $s,t \in \{0,1\}$, there exists a positive measure set of points $A_{st}\subset X_s \times X_t$ such that $d(x,y) > \delta$ for all $(x,y) \in A_{st}$. Under conservative ergodicity and nonsingularity, $\bigcup_{n \ge N} (S^2\times S^2)^n (B) = X_s \times X_t\mod{\mu}$ for any positive measure set $B_{ts}$ and any $N \in \N$. In particular, almost every point of $X_s \times X_t$ is sent to $A_{st}$ infinitely often, so indeed $S$ is W-measurably sensitive on $Y$. However, $S$ has a rotation on two points as a factor, so it fails to be weak mixing, and thus $L_\infty$-weak mixing, on the same set. 

\end{proof}

\subsection{Koopman Mixing}

It now follows by Theorem~\ref{T:dergodic2} that there exist $\Z$ actions that are EIC 
and with infinite conservative index but with non-ergodic cartesian square. This example 
is different from the examples in \cite{GlWe15}. We note that
in \cite{GlWe15}, the authors construct three examples that are EIC but with non-ergodic cartesian square (not doubly ergodic in the notation of \cite{GlWe15}).
The first example in \cite[3.5]{GlWe15} is a nonsingular action of a nonabelian group which is  SAT (a property that does not hold for nontrivial actions of abelian groups \cite[3.2]{GlWe15}), hence EIC, but not ergodic cartesian square. It is interesting to note that by  Lemma~\ref{L:kconserg}, the proof in \cite[3.5]{GlWe15} 
also shows the action is not weak doubly ergodic. Similarly, the example in \cite[5.1]{GlWe15} is a nonabelian action which is EIC and for which the proof in \cite[5.1]{GlWe15} shows it is not weak doubly ergodic. Thus these two examples are EIC and not weak doubly ergodic but for nonabelian actions. The third example in \cite[Proposition 7.1]{GlWe} is a $\Z$ action  $T$ that is EIC 
but is such that  $T\times T$ is not conservative, hence not ergodic, thus different from our example which has infinite conservative index.

We conclude this section with an example of an ergodic and Koopman mixing transformation that is not EIC.
An infinite measure-preserving  transformation $T$ is {\it Koopman mixing} if for all sets $A,B$ of finite measure 
$\lim_{n\to\infty}\mu(T^nA\cap B)=0$. This notion was defined by Hajian and Kakutani as zero-type, see \cite{EHIP14}. We recall that  the Koopam operator $U$ on $L^2$ is defined by $U(f)=f\circ T$; then mixing for a finite measure-preserving transformation is equivalent to the fact that $U^n$ converges weakly to 0 in the orthocomplement of the constants. When the measure is infinite, as 0 is the only constant in $L^2$, this condition is equivalent to Koopman mixing; this has been called mixing in other works, see e.g. \cite{DaRy11}. In rank-one transformations, the following property guarantees that a transformation is not EIC.

\begin{lemma} \label{L:notEIC}
Let $T$ be a rank-one transformation with height set elements that are all divisible by $k$, for some $k\ge 2$. Then $T$ is not weak mixing, and thus not EIC.
\end{lemma}
\begin{proof}
Define the function $L:X \rightarrow \{0,\ldots,k-1\}$ sending $x$ to the height, reduced modulo $k$, of the level containing $x$ in the first column $C_{i(x)}$ in which $x$ appears,. It follows that in all subsequent columns, $x$ appears in levels of height $L(x) \mod{k}$. Hence, $T$ has a rotation on $k$ elements as a factor.
\end{proof}

\begin{proposition}\label{L:mixingnotDE}
There exists an ergodic Koopman mixing infinite measure preserving transformation that is not $L_\infty$ weak mixing.
\end{proposition}
\begin{proof}
We use an ergodic rank-one transformation which uses all even height set elements. Take for instance the transformation which uses $H_n = \left\{ 0, 2h_n,4h_n,\ldots,2^{n+1} h_n\right\}$, where we add at least $(2^{n+1} + 1) h_n$ spacers on the last subcolumn (this number we select so that $h_{n+1}$ is even). Take any union of levels $B$ from column $C_m$, with descendant heights indexed by $D(B,n)$, for all $n \ge m$. 

We will show that it is impossible for $\mu\left( B \cap T^k(B) \right) \ge \frac{2}{n} \mu(B)$ when $k \ge h_n$ for some $n \ge m$. Indeed, suppose that this were the case: then we could write $k \in [h_n,h_{n+1})$ for such a $n$. Consider the column $C_{n+1}$, which contains $n + 2$ copies of $D(B,n)$ at heights given by $D_i(B,n+1) = 2^i h_n + D(B,n),\, i = 1,\ldots,n+1$ and $D_0(B,n+1) = D(B,n)$ (call their union $D(B,n+1)$). We can consider only $k \le h_{n+1}/2$, for if $k > h_{n+1}/2$ then in column $C_{n+2}$ each copy of $D(B,n+2)$ is sent entirely to the spacers comprising the upper half of $C_{n+1}$ or to the (at least) $h_{n+1}$ spacers added to each subcolumn of $C_{n+1}$ to produce $C_{n+2}$. 

For such a $k \le h_{n+1}/2$, if our supposition holds, it must be true that $\big| k + D(B,n+1) \cap D(B,n) \big| \ge \frac{2}{n} |D(B,n+1)|$, for we can express $B \cap T^k (B)$ entirely as a union of levels in $C_{n+1}$. Since $k \ge h_n$, any copy $D_i(B,n+1)$ cannot intersect itself under translation by $k$. Also, it is clear that any copy of $D(B,n)$ in $C_{n+1}$ can intersect at most one other copy. Suppose that $k + D_i(B,n+1) \cap D_j(B,n+1) \ne \emptyset$ for some $j > i \ge 1$. Then it must be the case that $ \big| 2^i + k - 2^j \big| h_n \le h_n$, whence $\big| 2^i + k - 2^j \big| \le 1$. Clearly, $k$ can only hold this property for one pair $(i,j)$ with $i,j \ge 1$, so there can only be one nonempty intersection of the form $k+ D_i(B,n+1) \cap D_j(B,n+1)$, $i,j \ge 1$. This leaves open the possibility that $k+D_0(B,n+1)$ intersects $D_i(B,n+1)$ for some positive $i$, but in any case we must have $\big| k + D(B,n+1) \cap D(B,n+1) \big| \le \frac{2}{n+2} |D(B,n+1)|$, a contradiction. So 
\[\lim_{k \rightarrow \infty} \mu\big( B \cap T^k(B) \big) < \lim_{k\rightarrow \infty} \frac{2}{n(k)} = 0,\]
where $n(k)$ is the unique value of $n$ such that $k \in [h_n,h_{n+1})$. By a simple approximation argument, like that employed in Theorem \ref{T:dergodic}, $T$ is Koopman mixing but not weak mixing by Lemma \ref{L:notEIC}.
\end{proof}

We now have in Theorem \ref{T:dergodic} a transformation which is weak doubly ergodic (thus weak mixing) but not Koopman mixing, and in this lemma a Koopman mixing transformation that is not $L_\infty$-weak mixing. It follows that weak double ergodicity (and $L_\infty$-weak mixing) is independent from Koopman mixing.

\section{Further Observations on WDE}

\subsection{WDE on a Partition}

For any $k \in \N$, Lemma \ref{L:notEIC} uses a transformation $T$ which has the shift by $1$ on $\Z_k$ as a factor (take the factor map $\psi$ which assigns each point to the height of the level $\mod{k}$ in which it first appears). Though $T$ is not EIC for $k \ge 2$, it is possible for $T$ to be WDE on each set $\psi^{-1} (z), \, z \in \Z_k$. With $k = 2$, this was essentially the case of Proposition 5.1 in \cite{BoFiMaSi01}, where $\psi^{-1}(0)$ was a set intersecting each interval in $\R$ in positive measure. Additionally, it is possible for $T$ to be power weakly mixing on each $\psi^{-1}(z), \, 0 \le z \le k-1$ but not EIC:
\begin{lemma}
For any $k \ge 2$, there exists a rank-one $T$ which is not $L_\infty$-weak mixing but is WDE on a partition of size $k$. There also exist rank-one transformations $T$ that are not $L_\infty$-weak mixing but are power weakly mixing on a partition of size $k$. 
\end{lemma}
\begin{proof}
Let $T$ be a staircase-type transformation using steps of size $k$, so $H_n = \left\{0,h_n,2h_n+k, \ldots , h_n + k \sum_{i=1}^{r_n-2}i\right\}$, all elements divisible by $k$. The partition is then the $k$ sets $S_\ell, \, 0\le \ell \le k-1$, containing the $x$ which appear in levels of height $\ell \mod{k}$. The proof of double ergodicity on each $S_\ell$ is almost identical to that of Proposition \ref{P:hardstairs} with the minor difference that differences between level heights must be multiplied by $k$. $T$ is not $L_\infty$-weak mixing by Lemma \ref{L:notEIC}. 

Similarly, one can take a power weakly mixing rank-one $T$ (e.g.\ see \cite{DGMS99}) and multiply all of its height set elements by $k$. Letting $S_\ell, \, \ell= 0,\ldots , k-1$ be defined as before, then $T^k$ is a rank-one transformation that is closed on $S_\ell$ with the same height sets as $T$, so it is power weakly mixing on each $S_\ell$ but not $L_\infty$-weak mixing. 
\end{proof}

\subsection{Invariant Sets when $T$ is not WDE}
From the definition of double ergodicity it is clear that if $T$ is not weak doubly ergodic, $T\times T$ is not ergodic (i.e.\ $(T\times T)^{-n} (A \times A ) $ does not intersect $A \times B$ for any positive $n$). What is less clear is what the neither null nor conull invariant sets for $T \times T$ are; with the following Lemma we can make this determination.

\begin{lemma}\label{L:deequivalence}
The following are equivalent if $T$ is invertible nonsingular ergodic: 
\begin{enumerate}
	\item[(I)] $T$ is weak doubly ergodic 
	\item[(II)] For every $A,B$ of positive measure in $X$ there exists some $n \in \Z\setminus \{0\}$ such that $\mu (A \cap T^n A) > 0$ and $\mu ( B \cap T^n A) > 0$
	\item[(III)] For every $A,B,C,D$ of positive measure in $X$ there exists some $n \in \Z\setminus \{0\}$ such that $\mu (A \cap T^n B) > 0$ and $\mu ( C \cap T^n D) > 0$.
\end{enumerate}
\end{lemma}
\begin{proof}
The implication $(I)\implies (II)$ is clear, so we only need to show the converse. Suppose for the sake of contradiction that there exist positive measure sets $A,B$ such that the only $n$ such that $\mu(A \cap T^n A) > 0$ and $\mu (B \cap T^n A) > 0$ are positive. First, we establish that these $n$ are bounded; by supposition we can find an $n_1$ such that $\mu(A \cap T^{n_1} A) > 0$ and $\mu (B \cap T^{n_1} A) > 0$, and this $n_1$ is strictly positive or else $\mu (A \cap B) > 0$ and any $n$ such that $\mu\big(T^{-n}(A \cap B) \cap (A \cap B)\big) > 0$ would be a contradiction. By $(II)$ there exists an $n_2$ such that 
\[ 
\mu \Big( (A \cap T^{n_1} A) \cap T^{n_2} (A \cap T^{n_1} A) \Big) > 0  \text{  and  } \mu \Big( (B \cap T^{n_1} A) \cap T^{n_2} (A \cap T^{n_1} A) \Big)>0,
\]
and again by supposition $n_2 > 0$. Thus, $\mu ( A \cap T^{n_1 + n_2} A) > 0$ and $\mu (B \cap T^{n_1 + n_2} A) > 0$. Continuing in this manner, the exponent can be made arbitrarily large. 

Now by ergodicity of $T$, there is some $i \in \Z$ such that $\mu (T^i A \cap B ) > 0$. By (II), by the assumption on $A$ and $B$, and by the result just deduced, the set of $n$ such that $\mu \big( ( A\cap T^{-i} B) \cap T^n (A \cap T^{-i} B )  \big) > 0$ and $\mu \big(( T^i A\cap  B) \cap T^n (A \cap T^{-i} B )  \big) > 0$ is both positive and unbounded. But for such an $n$, by application of $T^{i - n}$ to the first inequality $T^{-n}$ to the second, one obtains: 
\begin{align*}
	\mu \Big( T^{-n} (T^i A \cap (B)) \cap (T^i A \cap B) \Big)> 0 \text{  and  } \mu \Big( T^{-n} ( T^i A \cap B) \cap  (A \cap T^{-i} B) \Big) > 0,
\end{align*}
whence $\mu (  A \cap T^{i-n} A ) > 0$ and $\mu ( B \cap T^{i - n} A ) > 0 $. As $n$ can be chosen to be arbitrarily high, this is a contradiction. 

The implication (III)$\implies$(II) is clear. The converse is established by the proof of Proposition 2.1 in \cite{BoFiMaSi01} with the minor change that $\ell$ can be positive or negative. 
\end{proof}

Note that Lemma \ref{L:deequivalence} precludes an obvious example of EIC$\not\implies$DE, which is an invertible transformation $T$ for which there exist positive sets $A,B,C,D$ such that all $n$ such that $\mu\Big( (T\times T)^n(A \times B) \cap C \times D \Big)$ are strictly positive. Such a $T$ would be EIC by the same argument of Proposition \ref{P:WDE-EIC} but not WDE by definition. By the equivalence obtained above, we can exactly give some invariant positive not-full measure sets for $T\times T$ when $T$ is not WDE:

\begin{corollary} 
If $T$ is nonsingular ergodic, and $T$ is not weak doubly ergodic, then $\bigcup_{n \in \Z} (T\times T)^n (A \times A)\cap \bigcup_{n \in \Z} (T\times T)^n (A\times B) = \emptyset$ for positive measure sets $A,B$. 
\end{corollary}

\section{Appendix: $(1-q^{-1})$-type Transformations}
 A notion related to partial rigidity was defined in
 \cite{HaOs71} where a transformation is said to be of $\boldsymbol{\alpha}$-\textbf{type}, $0<\rho\leq 1$,  if for every finite measure set $A$, $\limsup_{n\rightarrow \infty} \mu(T^n A \cap A) = \alpha \mu(A)$. Examples of such transformations with $\alpha \in [0,1]$ are given in  \cite{HaOs71}. We will study these properties in conjunction with various ergodic properties. 
 We will be interested in $\alpha$-type transformations with $\alpha< 1$; these transformations are called  
 \textbf{rigidity-free} transformations in \cite{Ro09}. We note that rigidity-free is equivalent to $\liminf_{k \rightarrow \infty} \mu( T^{k} A \bigtriangleup A) > 0$ for all finite measure sets $A$ of positive measure. 
 
 We examine the partial rigidity properties of weak doubly ergodic transformations with $T\times T$ conservative but not ergodic. We will show that we can have a high degree of control over the partial rigidity of transformations without sacrificing these properties. The goal is to have a $\left(1-q^{-1}\right)$-type transformation which we denote $T_q$, for all integers $q\ge 2$. Since $\alpha$-type implies partial rigidity which in turn implies infinite conservative index, the transformation of this section is a refinement of that of section \eqref{S:DE}.

First, we define $T_q$ by its height set sequence. Fix a natural number $q\ge 2$. For $n$ even, set $H_n = \big\{0,\gamma_n,2\gamma_n,\ldots,(q-1)\gamma_n\big\}$, where $\gamma_n= 2h_n$; because we always have the inequality $h_n > \max D(I,n)$, this must imply $2h_n > 2\max D(I,n) + 1$. For $n$ even, $\max D(I,n+1)$ can be calculated explicitly as $\max D(I,n) + (q-1) \gamma_n$. Note that by Lemma \ref{L:two}, we can choose the number of cuts to be employed in the next height set $r_{n+1}$ so high that at least $1-\frac{1}{4(n+1)^2}$ of the pairs $\left(i,j\right)\in \{0,\ldots,r_{n+1}-1\}^2$ have $|i-j|>2\max D(I,n+1)+1$. For instance, by solving for the inequality in equation (\ref{E:pairsinequality}) for $m=2q \gamma_n$, we see that we can set 
\begin{equation}
r_{n+1}>2\left((2m-1)n^2+\sqrt{n^2-2m^2n^2+n^4-4mn^4+4m^2n^4}\right)\label{E:distanceinequality}
\end{equation}
to achieve the desired inequality. Then, add enough spacers on the rightmost subcolumn of column $C_n, \, n\text{ even}$ in order to get 
\begin{equation} \label{E:hnselection}
h_{n+1}>\max\left\{2\sum_{i=1}^{r_{n+1}-1} i + 2\max D(I,n+1)+1,\, 10r_{n+1},\, 10qh_n\right\}.
\end{equation}
Finally, set 
\[H_{n+1} = \left\{ 0, \, h_{n+1} +1,\, 2h_{n+1}+3,\, \ldots,\, (r_{n+1}-1) h_{n+1} + \sum_{i=1}^{r_n-1} i \right\}.\]

For $n$ odd, we write 
\[H_n = \big\{ \beta_0,\ldots,\beta_{r_n-1} \big\},\]
where $\beta_i = i h_n + x_i - 1$, and $x_i$ is the $i$th triangular number. 

Consider a set $B$ which is a finite union of levels in $C_{n-1}$, where $n$ is odd. Note that for any $j \ge -1$, $B$ can be written as a finite union of levels in $C_{n+ j}$. Call this finite union $D(B,n+j)$. In particular, observe that
\[ D(B,n+1) = H_n + D(B,n) = \bigcup_{i = 0}^{r_n-1} \beta_i + D(B,n).\]
Thus, there are $r_n$ copies of $D(B,n)$ in $D(B,n+1)$, which we denote by $D_i(B,n+1)= \beta_i + D(B,n)$. It is clear that the diameter of the set $D(B,n)$ is at most $(q-1)\gamma_{n-1} + h_{n-1} - 1 < 2qh_{n-1}$, but we always have $\beta_i - \beta_{i-1} > h_n > 10qh_{n-1}$. So the distance from the bottom of any one copy of $D(B,n)$ in $D(B,n+1)$ to the bottom of its adjacent copies is more than twice its diameter, and under nonzero translation of $D(B,n+1)$, any such copy can intersect with only one other copy. We now discuss the descendant sets of $T_q$.

Fix an integer $k > 0$. For brevity, define 
\begin{align*}
S_n = \Big\{i,\, 0 \le  i < r_n :\, k+ D_i(B,n+1) \, \cap \, D_j(B,n+1) \ne \emptyset& \\
\text{ for some } j, \,  0 \le i < r_n& \Big\}
\end{align*}
as the set of indices of translated copies of $D(B,n)$ having nonempty intersections with another $D(B,n)$-copy in $D(B,n_1)$. For any $k \ge 1$ and $i \in S_n$, we may also define the bijection $\phi_k(i)$ from $S_n$ to $\{0,\ldots,r_n\}$ sending $i \in S_n$ to the unique index $j$ satisfying $\big(k+ D_i(B,n+1)\big) \cap  D_j(B,n+1) \ne \emptyset$. For $n$ odd, this allows us to write intersections of $D(B,n+1)$ with its translates by $k$ as individual intersections of $D(B,n)$-copies with other copies:
\begin{align} 
 D(B,n+1) \cap \big( k + D(B,n+1)\big) = \bigsqcup_{i \in S_n} & \Big( k+ D_i(B,n+1) \nonumber \\
& \cap D_{\phi_k(i)}(B,n+1)\Big). \label{E:cupdecomp}
 \end{align}
 These definitions aid in the proof of the following lemma:

\begin{lemma} \label{L:firstodd}
Let $n$ be odd and $B$ a union of levels in $C_{n-1}$. Let $K \subset \N$ denote the set of positive integers $k$ such that $\Big|D(B,n) \cap \big(k + D(B,n)\big)\Big| > \left(1- q^{-1}\right) |D(B,n)|$. Then if $k \le 2qh_{n-1}$ and $k \not \in K$, $\Big|D(B,n+1) \cap \big(k + D(B,n+1)\big)\Big| \le \left(1- q^{-1}\right) |D(B,n+1)|$. 
\end{lemma}

\begin{proof}

For the first claim, fix an integer $k \in K^c \cap [1,\ldots,2qh_{n-1}]$. Then $k$ is too small to bridge the distance between adjacent $D(B,n)$-copies in $D(B,n+1)$, so for any $i \in \{0,\ldots,r_n - 1\}$, either $i \not \in S_n$ or $\phi_k(i) = i$. Because $k \not\in K$, each self-intersection $D_i(B,n+1) \cap \big( k + D_i(B,n+1)\big)$ has order at most $\left( 1- q^{-1} \right) \big| D(B,n+1)\big|$. By equation (\ref{E:cupdecomp}), 
\begin{align*}
\Big| D(B,n+1) \cap \big( k + D(B,n+1)\big) \Big| &\le |S_n| \left( 1- q^{-1} \right) \big| D(B,n+1)\big| \\
&\le \left( 1- q^{-1}\right) \big| D(B,n+1)\big|.
\end{align*}
\end{proof}

We now require some notation to deal with elements of $S_n$. Fix an integer $k > 2qh_{n-1}$. We will partition $S_n$ into the disjoint union of sets $I_1 \sqcup I_2 \sqcup \ldots \sqcup I_p$ as follows: let $i_1$ be the minimal element of $S_n$, and let $I_1 = \big\{ i \in S_n:\, \phi_k(i) - i = \phi_k(i_1) - i_1\big\}$. Then let $i_2$ be the minimal element of $S_n \setminus I_1$ (if it exists), and let $I_2 = \big\{ i \in S_n:\, \phi_k(i) - i = \phi_k(i_2) - i_2\big\}$. Proceed until the sets $I_1 \sqcup \ldots \sqcup I_p$ form a partition for $S_n$. For all $\ell \in \{1,\ldots,p\}$, let $j_\ell = \phi_k(i_\ell)$.

Fix an $\ell \in \{1,\ldots,p\}$, and an $i \in I_\ell$. Then we should have $\phi_k(i_\ell) - i = j_\ell - i_\ell$. Set $z = i - i_\ell = \phi_k(i) - j_\ell$. Note that for any $z \in \Z$ with $-i_\ell \le z < r_n - j_\ell$,
\begin{align} 
\beta_{\phi_k(i)} - \beta_i &= \beta_{j_\ell + z} - \beta_{i_\ell + z} =(j_\ell - i_\ell)h_n + x_{j_\ell} - x_{i_\ell} + (j_\ell - i_\ell)z \\
&= \beta_{j_\ell} - \beta_{i_\ell} + (j_\ell - i_\ell) z.\label{E:zeqn}
\end{align}
Because $k > 2qh_{n-1}$, which exceeds the diameter of $D(B,n)$, it is impossible for the $k$-translated copy $D_i(B,n)$ to intersect itself, so $j_\ell - i_\ell> 0$. Crucially, the diameter of $D(B,n)$ also implies that for every $i' \in S_n$, $\phi_k(i')$ must satisfy 
\begin{equation} \label{E:crucialeqn}
-2qh_{n-1} < \beta_{\phi_k(i')} - \beta_{i'} - k < 2qh_{n-1}.
\end{equation}
By application of (\ref{E:zeqn}) and (\ref{E:crucialeqn}), for every $z \in \big\{ i' - i_\ell,\, i' \in I_\ell \big\}$, the following inequality must hold:
\begin{equation}  \label{E:I1restrict}
-2qh_{n-1} < \beta_{j_\ell} - \beta_{i_\ell} + (j_\ell - i_\ell) z - k < 2qh_{n-1}.
\end{equation}

\begin{lemma} \label{L:secondodd}
Let $n$ be odd and $B$ be a union of levels in $C_{n-1}$. Let $S_n = I_1 \sqcup \ldots \sqcup I_p$, as previously described. If $k > 2qh_{n-1}$, $\Big|D(B,n+1) \cap \big(k + D(B,n+1)\big)\Big| \le \left(1- q^{-1}\right) |D(I,n+1)|$.
\end{lemma}

\begin{proof}

The proof of the lemma is trivial in the case where $S_n = \emptyset$, so suppose first that $I_1$ is nonempty and $p = 1$. 
The bounds in (\ref{E:I1restrict}) imply that $|S_n| = |I_1| < \frac{4qh_{n-1}}{j_1 - i_1} \le 4q h_{n-1}$. But by equation (\ref{E:distanceinequality}), $r_n > 8 q \gamma_n = 16 q h_{n-1}$. Thus, $|S_n| < \frac{r_n}{2}$, so use of (\ref{E:cupdecomp}) implies:
\begin{align} 
\Big|D(B,n+1) \cap \big(k + D(B,n+1)\big)\Big| \le |S_n| \, |D(B,n)| < \frac{r_n}{2} \, |D(B,n)| \nonumber \\
\le \left( 1- q^{-1}\right) |D(B,n+1)|.\label{E:qeqn}
\end{align}

Now suppose that $p > 1$. By (\ref{E:I1restrict}), $|I_p| < \frac{4qh_{n-1}}{j_p - i_p}$. Also, for any $z \in \Z$ with $-i_p \le z < r_n - j_p$, we may write 
\[
\beta_{j_p + z} - \beta_{i_p + z} = \beta_{j_p} - \beta_{i_p} + (j_p- i_p)z.
\]

In the case where $z \le -\frac{4qh_{n-1}}{j_p - i_p}$, Whenever $j \le j_p$, (\ref{E:zeqn}) implies that 
\begin{align*}
\beta_{j + z}& - \beta_{i_p + z} - k \le \beta_{j_p + z} - \beta_{i_p + z} - k  \\
&\le \beta_{j_p} - \beta_{i_p} - (j_p - i_p) \, \frac{4qh_{n-1}}{j_p - i_p} - k < -2qh_{n-1}.
\end{align*}
Such $z$ and $j$ do not satisfy (\ref{E:crucialeqn}), so $k + D_{i_p+z}(B,n+1)$ has empty intersection with all $D(B,n)$-copies $D_{j+z}(B,n+1)$ with $j \le j_p$. In addition, when $z \ge -\frac{9r_n}{j_p - i_p}$ and $j \ge j_p + 1$, then we should have 
\begin{align*}
\beta_{j + z } - \beta_{i_p + z} - k \ge  \beta_{j_p} + h_n - \beta_{i_p} - k - (j_p - i_p) \, \frac{9r_n}{j_p - i_p} \\
> \beta_{j_p} - \beta_{i_p} - k + r_n > -2qh_{n-1} + 16qh_{n-1} > 2qh_{n-1},
\end{align*}
So when $ z \ge -\frac{9r_n}{j_p - i_p}$, $k + D_{i_p + z} (B,n+1)$ cannot intersect $D_{j+ z} (B,n+1) $ for indices $ j \ge j_p + 1$. We conclude that for all 
\[
i \in \left(i_p + \left[ -\frac{9r_n}{j_p - i_p}, \, -\frac{4qh_{n-1}}{j_p - i_p}\right] \cap \big\{0,\ldots,r_n - 1\big\}\right),
\]
$k + D_{i_p + z} (B,n+1)$ does not intersect with any $D(B,n)$-copy in $D(B,n+1)$.

We now show that $i_p - \frac{9r_n}{j_1 - i_1} > 0$. By supposition, $i_{p-1} \in S_n$ is a nonnegative integer less than $i_p$ such that $\phi_k(i_{p-1}) - i_{p-1} \ne \phi_k(i_p) - i_p$. Set $z = i_{p-1} - i_{p}$. Suppose for the sake of contradiction that $z \ge -\frac{9r_n}{j_p - i_p}$: then, as we have already deduced, $k + D_{i_{p-1}} (B,n+1) = k+ D_{i_p + z} (B,n+1)$ cannot intersect $D_{j + z} (B,n+1)$ whenever $j \ge j_p + 1$. We also find that for $j \le j_p - 1$,
\begin{align*}
\beta_{j + z} - \beta_{i_p + z} - k &\le \beta_{j_p + z} - h_n - \beta_{i_p + z} - k \\
&\le \beta_{j_p} - h_n - \beta_{i_p} - k - (j_p - i_p) (i_{p} - i_{p-1}) \\
&< \beta_{j_p} - \beta_{i_p} - k - h_n < 2qh_{n-1} - h_n < -8 qh_n.
\end{align*}
So $D_{i_{p-1}} (B,n+1)$ has empty intersection with all sets $D_{j + z} (B,n+1)$ with $j \le j_p - 1$. But this is impossible, because then we must have $\phi_k(i_{p-1}) = j_p + z = j_p + i_{p-1} - i_p$, implying that $\phi_k(i_{p-1}) - i_{p-1} = j_p - i_p$, which is a contradiction. So $i_{p-1} < i_p - \frac{9r_n}{j_p - i_p}$; because $i_{p-1} \ge 0$, we conclude that $i_p + \left[ -\frac{9r_n}{j_p - i_p}, \, -\frac{4qh_{n-1}}{j_p - i_p}\right] \subset [0,\ldots,r_n-1]$. Because $z = i_{p-1} - i_p < \frac{-9r_n}{j_p - i_p} < \frac{-4h_{n-1}}{j_p - i_p}$, it also must be true that $k + D_{i_{p-1}}(B,n+1)$ cannot intersect any $D(B,n)$-copy $D_{j + z} (B,n+1)$ with $ j \le j_p$; hence, $j_{p-1} = \phi_k(i_{p-1}) < j_p + i_{p-1} - i_p$, and $j_{p-1} - i_{p-1} < j_p - i_p$. 

By supposition, $p \ge 2$, so $I_{p-1}$ must be nonempty. From application of (\ref{E:I1restrict}), we have the strict upper bound $|I_{p-1}| < \frac{4qh_{n-1}}{j_{p-1} - i_{p-1}} < \frac{4qh_{n-1}}{j_{p} - i_{p}}$. On the other hand, the set $\left(i_p + \left[ -\frac{9r_n}{j_p - i_p}, \, -\frac{4qh_{n-1}}{j_p - i_p}\right]\right) \cap \Z$ lies strictly between $i_{p-1}$ and $i_p$, is contained in $\{0,\ldots,r_n-1\} \setminus S_n$, and has order at least, say, $\frac{7r_n}{j_p - i_p} > 10 \,\frac{4qh_{n-1}}{j_{p} - i_{p}} > 2\big( |I_{p-1}| + |I_p| \big)$. Similarly, if $i_{p-2} < i_{p-1}$ exists, there must exist a subset of $\{0,\ldots,r_n-1\} \setminus S_n$ lying strictly between $i_{p-2}$ and $i_{p-1}$ with order exceeding $2 \big( |I_{p-2} + |I_{p-1}| \big)$. So clearly, we must have $|S_n| < \frac{r_n}{2}$, from which (\ref{E:qeqn}) can be applied. 

\end{proof}

\begin{lemma} \label{L:even}
Let $n$ be even and $B$ a union of levels in $C_{n}$. Let $K \subset \N$ denote the set of positive integers $k$ such that 
\[
\Big|D(B,n) \cap \big(k + D(B,n)\big)\Big| > \left(1- q^{-1}\right) |D(B,n)|.
\]
Then if $k \not \in K$, $\Big|D(B,n+1) \cap \big(k + D(B,n+1)\big)\Big| \le \left(1- q^{-1}\right) |D(B,n+1)|$. 
\end{lemma}

\begin{proof}
By selection of the even height sets, 
\[
D(B,n+1) = \big\{D(B,n),\, \gamma_n + D(B,n),\ldots,\, (q-1)\gamma_n + D(B,n)\big\}.
\]
For $n$ even and $i \in \{0,\ldots,q-1\}$, let $D_i(B,n+1) = i(\gamma_n) + D(B,n)$. Since $\gamma_n = 2h_n > 2\max D(B,n)$, any translation of $D_i(B,n+1)$ by $k$ can intersect with at most one other copy $D_j(B,n+1),\, j \in \{0,\ldots,q-1\}$.

Fix an integer $k \ge 1$. If $k \le \max D(B,n)$, then $ k + D_i(B,n)$ can only intersect with $D_i(B,n)$ (the same subcopy) for $i = 0,\ldots,q-1$. If $k \not\in K$, each such intersection must have size at most $\left(1- q^{-1}\right) |D(B,n)|$; hence, $\Big|D(B,n+1) \cap \big( k + D(B,n+1)\big) \Big|  \le q \left(1- q^{-1}\right) |D(B,n)| = \left(1- q^{-1}\right)|D(B,n+1)|$.

On the other hand, if $k > \max D(B,n)$, then the bottom level of $ k + D_{q-1} (B,n+1)$ is sent above the highest level of $D(B,n+1)$ (namely, the top level of $D_{q-1} (B,n+1)$). Thus, at least one of the translated $D(B,n)$-copies in $D(B,n+1)$ has an empty intersection with all other such copies, and as such,
\begin{align*}
\Big|D(B,n+1) \cap \big( k + D(B,n+1)\big) \Big| & \le (q-1)|D(B,n)| \\
&= \left( 1- q^{-1}\right) |D(B,n+1)|.
\end{align*}
\end{proof}

\begin{lemma}  \label{L:nonrigid}
Let $B$ be any union of levels drawn from $C_m,\, m \in \N$. Then the set of integers 
\[K = \left\{k \in \N: \, \mu\left( B \cap T_q^k (B)\right)> \left( 1- q^{-1}\right) \mu(B)\right\}\]
is finite.
\end{lemma}
\begin{proof}
Without loss of generality we can assume that $m$ is even, considering the descendants of $B$ in a subsequent column, if necessary. Let $D(B,m)$ be the set of heights of the levels comprising $B$ in $C_m$, and let $K^*$ be the set of positive integers $k$ allowing $\Big|D(B,m) \cap \big(k + D(B,m)\big)\Big| > \left(1- q^{-1}\right) |D(B,m)|$. Clearly, $K^*$ is upper bounded by $\max D(B,m)$, whence it is finite. By Lemma \ref{L:even}, the only positive integers $k$ allowing $\Big|D(B,m+1) \cap \big(k + D(B,m+1)\big)\Big| > \left(1- q^{-1}\right) |D(B,m+1)|$ belong to $K^*$. By Lemma \ref{L:secondodd}, the only integers $k$ allowing $\Big|D(B,m+2) \cap \big(k + D(B,m+2)\big)\Big| > \left(1- q^{-1}\right) |D(B,m+2)|$ are those less than $2qh_{m}$, and by Lemma \ref{L:firstodd}, these integers all must belong to $K^*$. By inductively applying Lemmas \ref{L:even}, \ref{L:firstodd}, and \ref{L:secondodd}, we find that for any $n \in \N,\, n \ge m$, $\Big|D(B,n) \cap \big(k + D(B,n)\big)\Big| > \left(1- q^{-1}\right) |D(B,n)|$ if and only if $k \in K^*$, which is a finite set.

Now fix any integer $k \ge 1$ with $k \not\in K^*$. Because $T_q$ always adds spacers on its rightmost subcolumns, we can find an $n \in \N$ such that $h_n >\max D(B,n) + k$. Letting $I_n$ be the base of $C_n$, this implies that $T^j I_n$ is defined as a level in $C_n$ for every $j \in k + D(B,n)$. Thus,
\[T^k B = T^k \left(\bigsqcup_{d \in D(B,n)} T^d I_n\right) = \bigsqcup_{d \in k + D(B,n)} T^d I_n.\]
But $B = \bigsqcup_{d \in D(B,n)} T^d I_n$, and 
\[
\big| k + D(B,n) \cap D(B,n)\big| \le \left(1- q^{-1} \right) \big| D(B,n)\big|.
\]
Because all of the levels of $C_n$ are pairwise disjoint, it follows that $\mu(B \cap T^k B) \le \left( 1- q^{-1}\right) \mu(B)$. So the statement of the lemma holds with $K = K^*$. 
\end{proof}

The following is proved with the same argument as Lemma \ref{L:tqnotergodic}:

\begin{lemma}\label{L:tqqnotergodic}
For any $q \ge 2$, $T_q \times T_q$ is not ergodic.
\end{lemma}

Thus, we have the following extension of Theorem \ref{T:dergodic}:

\begin{theorem} For any $q\in \N$, $q\ge 2$, there exists a $\left( 1 - q^{-1} \right)$-type transformation $T_q$ such that, for any finite measure set $A$, we  have 
\[
\limsup_{n \rightarrow \infty} \mu\left( A \cap T_q^n(A)\right) = \left( 1- q^{-1}\right) \mu(A) < \mu(A).
\]
Furthermore, $T_q$ is weak doubly ergodic but $T_q\times T_q$ is not ergodic. \label{T:dergodic3}
\end{theorem}

\begin{proof}
$T_q$ is weak doubly ergodic by Proposition \ref{P:hardstairs} and at least $\left( 1- q^{-1}\right)$-partially rigid by application of Proposition \ref{P:partialrigid} to the even height sets $H_{2n},\, n \in \N$. In addition, $T_q \times T_q$ is not ergodic by Lemma \ref{L:tqqnotergodic}. To show the second claim, suppose for the sake of contradiction that for some $\ve>0$, there exists a finite positive measure set $A$ such that $\limsup_{n\rightarrow \infty} \mu( A\cap T^n (A)) \ge \left( 1- q^{-1} + \ve \right) \mu(A)$. We may approximate $A$ with a set $B$ constructed as a union of levels in some column of $T_q$ such that $\mu(A \triangle B) < \frac{\ve}{8} \, \mu(B)$. It then follows that $\mu(A) \ge \left( 1- \frac{\ve}{8}\right) \, \mu(B)$. By Lemma \ref{L:nonrigid}, there must exist some $N \in \N$ such that for all $n \ge N$, $\mu(B \cap T^n(B)) \le \left( 1- q^{-1} \right) \mu(B)$. By supposition, there must exist some $m \ge N$ such that $\mu(A \cap T^m(A)) > \left( 1- q^{-1} + \frac{\ve}{2} \right) \mu(A)$. But then 
\begin{align*}
\mu(B \cap T^m (B)) \ge \mu(A \cap T^n (A)) - 2\mu(A \setminus B) &> \left( 1- q^{-1} + \frac{\ve}{2} \right) \mu(A) - \frac{\ve}{4}\, \mu(B) \\
&> \left( 1- q^{-1} + \frac{\ve}{8} \right) \mu(B),
\end{align*}  
which is a contradiction. So in fact 
\[
\limsup_{n\rightarrow \infty} \mu( A\cap T^n (A)) = \left( 1- q^{-1} \right) \mu(A)
\]
for every finite measure set $A$. 
\end{proof}

One question that arises from Theorem \ref{T:dergodic} is whether $\alpha$-type transformations with $\alpha < 1$ are weak doubly ergodic. We answer in the negative, with the following lemma.

\begin{lemma}
For all $q \in \N$ that are at least $2$, there exists an infinite measure preserving $\left(1-q^{-1}\right)$-type transformation $T$ on $X$ which is not EIC, hence not weak doubly ergodic.
\end{lemma}
\begin{proof}
Construct $T$ with $
H_n = \{0,2h_n,\ldots,2(q-1)h_n\},
$ for all $n$. 
Always place at least one spacer on the rightmost subcolumn of $C_n$. By Proposition \ref{P:partialrigid}, it is clear that $T$ is at least $\left( 1- q^{-1}\right)$-partially rigid. By the method of Lemma \ref{L:even}, it is straightforward to show that there is a finite set $K \subset \Z$ such that if $k \not\in K$, $\big|D(B,n) \cap (k + D(B,n))\big| \le \left( 1- q^{-1} \right)\, |D(B,n)|$ for any $B$ a collection of levels from $C_i$ and $n \ge i$. Thus, by the argument given in Theorem \ref{T:dergodic}, $T$ is of $\left( 1- q^{-1} \right)$-type. Finally, $T$ is not EIC by Lemma \ref{L:notEIC}. 
\end{proof}
\bibliographystyle{plain}
\bibliography{ErgBib_Master}

\def\cprime{$'$}
\begin{thebibliography}{10}

\bibitem{Aa97}
J.~Aaronson.
\newblock {\em An introduction to infinite ergodic theory}, volume~50 of {\em
  Mathematical Surveys and Monographs}.
\newblock American Mathematical Society, Providence, RI, 1997.

\bibitem{Aa13p}
J.~Aaronson.
\newblock Priviate communication.
\newblock 2013.

\bibitem{AaLiWe79}
Jonathan Aaronson, Michael Lin, and Benjamin Weiss.
\newblock Mixing properties of {M}arkov operators and ergodic transformations,
  and ergodicity of {C}artesian products.
\newblock {\em Israel J. Math.}, 33(3-4):198--224 (1980), 1979.
\newblock A collection of invited papers on ergodic theory.

\bibitem{AdFrSi97}
Terrence Adams, Nathaniel Friedman, and Cesar~E. Silva.
\newblock Rank-one {W}eak {M}ixing for {N}onsingular {T}ransformations.
\newblock {\em Israel J. Math.}, 102:269--281, 1997.

\bibitem{AdSi15}
Terrence Adams and Cesar~E. Silva.
\newblock Weak rational ergodicity does not imply rational ergodicity.
\newblock {\em Israel J. Math, to appear}, 2015.

\bibitem{Ad98}
Terrence~M. Adams.
\newblock Smorodinskys conjecture on rank-one mixing.
\newblock {\em Proc. Amer. Math. Soc.}, 126(3):739--744, 1998.

\bibitem{BaYa15}
Rachel~L. Bayless and Kelly~B. Yancey.
\newblock Weakly mixing and rigid rank-one transformations preserving an
  infinite measure.
\newblock {\em New York J. Math.}, 21:615--636, 2015.

\bibitem{BeRo88}
Vitaly Bergelson and Joseph Rosenblatt.
\newblock Mixing actions of groups.
\newblock {\em Illinois J. Math.}, 32(1):65--80, 1988.

\bibitem{BoFiMaSi01}
A.~Bowles, L.~Fidkowski, A.~E. Marinello, and C.~E. Silva.
\newblock Double ergodicity of nonsingular transformations and infinite
  measure-preserving staircase transformations.
\newblock {\em Illinois J. Math.}, 45(3):999--1019, 2001.

\bibitem{Small2000}
Amie Bowles, Lukasz Fidkowski, Amy~E. Marinello, and Cesar~E. Silva.
\newblock Double {E}rgodicity of {N}onsingular {T}ransformations and {I}nfinite
  {M}easure-{P}reserving {S}taircase {T}ransformations.
\newblock {\em Illinois Journal of Mathematics}, 45, 2001.

\bibitem{Small2014}
Julien Clancy, Rina Friedberg, Indraneel Kasmalkar, Isaac Loh, Tudor
  P\u{a}durariu, Cesar~E. Silva, and Sahana Vasudevan.
\newblock Ergodicity and {C}onservativity of {P}roducts of {I}nfinite
  {T}ransformations and {T}heir {I}nverses.
\newblock {\em http://arxiv.org/abs/1408.2445}, 2014.

\bibitem{DGPS15}
Irving Dai, Xavier Garcia, Tudor P{\u{a}}durariu, and Cesar~E. Silva.
\newblock On rationally ergodic and rationally weakly mixing rank-one
  transformations.
\newblock {\em Ergodic Theory Dynam. Systems}, 35(4):1141--1164, 2015.

\bibitem{DaRy11}
Alexandre~I. Danilenko and Valery~V. Ryzhikov.
\newblock Mixing constructions with infinite invariant measure and spectral
  multiplicities.
\newblock {\em Ergodic Theory Dynam. Systems}, 31(3):853--873, 2011.

\bibitem{DaSi09}
Alexandre~I. Danilenko and Cesar~E. Silva.
\newblock Ergodic {T}heory: {N}onsingular {T}ransformations.
\newblock In {\em Encyclopedia of Complexity and System Science}, volume Part
  5, pages 3055--3083. Springer, 2009.

\bibitem{DGMS99}
Sarah~L. Day, Brian~R. Grivna, Earle~P. McCartney, and Cesar~E. Silva.
\newblock Power weakly mixing infinite transformations.
\newblock {\em New York J. Math.}, 5:17--24 (electronic), 1999.

\bibitem{EHIP14}
Stanley Eigen, Arshag Hajian, Yuji Ito, and Vidhu Prasad.
\newblock {\em Weakly wandering sequences in ergodic theory}.
\newblock Springer Monographs in Mathematics. Springer, Tokyo, 2014.

\bibitem{GlWe15}
Eli Glasner and Benjamin Weiss.
\newblock Weak mixing properties for non-singular actions.
\newblock {\em Ergodic Theory and Dynamical Systems}, FirstView:1--15, 7 2015.

\bibitem{GlWe}
Eli Glasner and Benjamin Weiss.
\newblock Weak mixing properties for non-singular actions.
\newblock {\em Ergodic Theory Dynam. Systems}, DOI:
  http://dx.doi.org/10.1017/etds.2015.16 (About DOI).

\bibitem{GrIo12}
Ilya Grigoriev, Marius~C{\u{a}}t{\u{a}}lin Iordan, Amos Lubin, Nathaniel Ince,
  and Cesar~E. Silva.
\newblock On {$\mu$}-compatible metrics and measurable sensitivity.
\newblock {\em Colloq. Math.}, 126(1):53--72, 2012.

\bibitem{KaPa63}
S.~Kakutani and W.~Parry.
\newblock Infinite measure preserving transformations with ``mixing''.
\newblock {\em Bull. Amer. Math. Soc.}, 69:752--756, 1963.

\bibitem{HaOs71}
Motosige Osikawa and Toshihiro Hamachi.
\newblock On zero type and positive type transformations with infinite
  invariant measures.
\newblock {\em Mem. Fac. Sci. Kyushu Univ. Ser. A}, 25:280--295, 1971.

\bibitem{Ro09}
Emmanuel Roy.
\newblock Poisson suspensions and infinite ergodic theory.
\newblock {\em Ergodic Theory Dynam. Systems}, 29(2):667--683, 2009.

\bibitem{Si08}
C.~E. Silva.
\newblock {\em Invitation to {E}rgodic {T}heory}, volume~42 of {\em Student
  Mathematical Library}.
\newblock American Mathematical Society, Providence, RI, 2008.

\end{thebibliography}
\end{document}